\documentclass[11pt]{amsart}

\usepackage{enumerate,url,amssymb,  mathrsfs, graphicx, pdfsync, amsthm}
\newtheorem{theorem}{Theorem}[section]

\newtheorem*{lemma*}{Lemma}

\newtheorem{proposition}[theorem]{Proposition}

\theoremstyle{definition}
\newtheorem{definition}[theorem]{Definition}

\newtheorem{question}[theorem]{Question}

\theoremstyle{remark}

\numberwithin{equation}{section}

\newcommand{\yy}{\mathbb{Y}}

\newcommand{\dd}{\mathbb{D}}

\newcommand{\abs}[1]{\lvert#1\rvert}

\newcommand{\C}{\mathbb{C}}

\newcommand{\W}{\mathscr{W}}

\newcommand{\E}{\mathcal{E}}

\newcommand{\R}{\mathbb{R}}

\newcommand{\Y}{\mathbb{Y}}

\newcommand{\dtext}{\textnormal d}

\newcommand{\onto}{\xrightarrow[]{{}_{\!\!\textnormal{onto\,\,}\!\!}}}

\DeclareMathOperator{\diam}{diam}

\DeclareMathOperator{\dist}{dist}

\DeclareMathOperator{\loc}{loc}

\def\leq{\leqslant}
\def\geq{\geqslant}

\def\le{\leqslant}
\def\ge{\geqslant}

\def\XXint#1#2#3{{\setbox0=\hbox{$#1{#2#3}{\int}$}\vcenter{\hbox{$#2#3$}}\kern-.5\wd0}}

\def\XXiint#1#2#3{{\setbox0=\hbox{$#1{#2#3}{\iint}$}\vcenter{\hbox{$#2#3$}}\kern-.5\wd0}}

\begin{document}
\title[Sobolev homeomorphic extensions onto John domains]{Sobolev homeomorphic extensions\\ onto John domains}


\author[P. Koskela]{Pekka Koskela}
\address{Department of Mathematics and Statistics, P.O.Box 35 (MaD) FI-40014 University of Jyv\"askyl\"a, Finland}
\email{pekka.j.koskela@jyu.fi}

\author[A. Koski]{Aleksis Koski}
\address{Department of Mathematics and Statistics, P.O.Box 35 (MaD) FI-40014 University of Jyv\"askyl\"a, Finland}
\email{aleksis.t.koski@jyu.fi}

\author[J. Onninen]{Jani Onninen}
\address{Department of Mathematics, Syracuse University, Syracuse,
NY 13244, USA and  Department of Mathematics and Statistics, P.O.Box 35 (MaD) FI-40014 University of Jyv\"askyl\"a, Finland
}
\email{jkonnine@syr.edu}
\thanks{P. Koskela was supported by the Academy of Finland Grant number 323960.  A. Koski was supported by the Academy of Finland Grant number 307023.
J. Onninen was supported by the NSF grant  DMS-1700274.}

\subjclass[2010]{Primary 46E35, 58E20}


\keywords{Sobolev homeomorphisms, Sobolev extensions, John domains, Quasidisks}

\begin{abstract} 
Given the planar unit disk as the source and a Jordan domain as the target, 
we study the problem of extending a given boundary homeomorphism as a Sobolev homeomorphism. For general targets, this Sobolev variant of the classical Jordan-Sch\"oenflies theorem may admit no solution - it is possible to have a boundary homeomorphism which admits a continuous $\W^{1,2}$-extension but not even a homeomorphic $\W^{1,1}$-extension. We prove that if the target is assumed to be a John disk, then any boundary homeomorphism from the unit circle admits a Sobolev homeomorphic extension for all exponents $p<2$.  John disks, being one sided quasidisks, are of fundamental importance in Geometric Function Theory.
\end{abstract}

\maketitle
\section{Introduction}

Throughout this text $\Y$ is a bounded Jordan domain and $\mathbb D$ is  the unit disk  in the complex plane $\mathbb C$. The classical Jordan-Sch\"oenflies theorem states that every homeomorphism $\varphi \colon  \partial \mathbb D \onto \partial \Y$ admits a continuous extension $h \colon \overline{\mathbb D }\to  \overline{\Y}$ which takes $\dd$ homeomorphically onto $\Y$. We are seeking for its Sobolev variant. 

\begin{question}\label{existQuestion}  Under which condition on $\Y$ does an arbitrary boundary homeomorphism $\varphi \colon \partial \dd \onto \partial \Y$ admit a homeomorphic extension $h \colon \overline {\dd} \onto \overline{\Y}$
of Sobolev class $\W^{1,p} (\dd, \C)$?
\end{question}
The most immediate  reason for studying such   a variant comes from the variational approach to Geometric Function Theory~\cite{AIMb, HKb, IMb, Reb} and  Nonlinear Elasticity~\cite{Anb, Bac, Cib}. Both theories share the compilation ideas to determine the infimum of  a given stored energy functional 
among  Sobolev homeomorphisms. When one studies such variational problems in the pure displacement setting, the first step is to ensure the existence of 
admissible homeomorphisms; that is, to answer Question~\ref{existQuestion}.  To begin, the boundary homeomorphism $\varphi$ must be the Sobolev trace of some (possibly non-homeomorphic) mapping in $\W^{1,p}(\dd, \C)$. Hence the best Sobolev regularity one can hope for is $p<2$ in Question~\ref{existQuestion}, see \cite{Ve}. On the other hand,
a Sobolev homeomorphic extension does not always exist for an arbitrary target domain even for a fairly regular boundary mapping. Indeed, there exists a Jordan domain $\Y$ and a homeomorphism $\varphi \colon \partial \mathbb D \onto \partial \Y $ which admits  a continuous $\W^{1,2}$-Sobolev extension to $\mathbb D$ but does not admit any homeomorphic extension to $\mathbb D$ in $\W^{1,1} (\mathbb D, \mathbb C)$, see~\cite{KOext, Zh}. Secondly, the requested $\W^{1,p}$-Sobolev homeomorphism in Question~\ref{existQuestion} exists for all  $p<2$ if the boundary of $\Y$ is rectifiable, see~\cite{KOext}. However, many important classes of domains studied in Geometric Function Theory include domains with nonrectifiable boundaries.  
Quasidisks serve as a standard example of such domains. A planar domain is a {\it quasidisk} if it is the image of an open disk  under a quasiconformal self mapping of $\mathbb C$, see Definition~\ref{def:quasi}. They have been studied intensively for many years because of their exceptional function theoretic properties, relationships with Teichm\"uller theory and Kleinian groups and interesting applications in complex dynamics, see~\cite{GeMo} for an elegant  survey. In particular, the Koch snowflake reveals the possible complexity of a quasidisk.
\begin{figure}[h]
\includegraphics[scale=0.2]{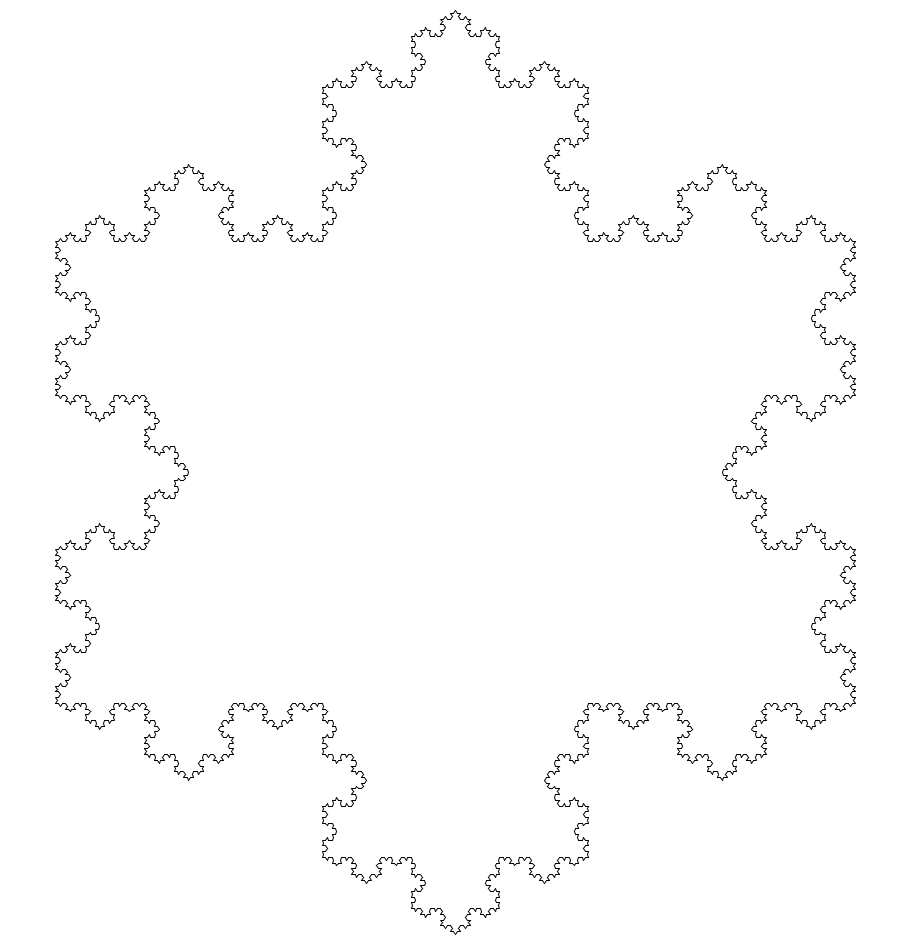}
\caption{The Koch snowflake reveals the complexity of a quasidisk.}
\label{cuspFig}
\end{figure}

\begin{theorem}\label{thm:quasidisk}
Let $\Y$ be a quasidisk and $\varphi \colon \partial \mathbb D \onto \partial \Y$ a homeomorphism. Then there exists a homeomorphic extension $h \colon \overline{\mathbb D}  \onto \overline{\mathbb Y } $ of $\varphi$ in $\W^{1,p} (\mathbb D, \C)$ for all $1\le p <2$.
\end{theorem}
Our argument generalizes to {\it John disks}, see Definition~\ref{def:john}.  A John disk is a  simply connected John domain.  Such domains may be regarded as one-sided quasidisks. They appear in many contexts in analysis~\cite{Br, Br2, CP,  GNV, GHM, MS, Mc,   Va}. John domains  were introduced by F. John~\cite{John} in connection with his work on elasticity.   Roughly speaking, a domain is a John domain if it is possible to travel from one point of the domain to another without going too close to the boundary. John domains are allowed to have inward cusps but not outward cusps.

\begin{theorem}\label{thm:johndisk}
Let $\Y$ be a John disk and $\varphi \colon \partial \mathbb D \onto \partial \Y$ a homeomorphism. Then there exists a homeomorphic extension $h \colon \overline{\mathbb D} \onto \overline { \mathbb Y}  $ of $\varphi$ in $\W^{1,p} (\mathbb D, \C)$ for all $1\le p <2$.
\end{theorem}

The key in our proofs  is the construction the following 
self-homeomorphic extension. 

\begin{theorem}\label{thm:weightedextension}
Let $1\le p <2$ and $p\beta <1$. Then every circle homeomorphism $\varphi \colon \partial \mathbb D \onto \partial \mathbb D$ has a locally Lipschitz 
continuous, homeomorphic extension $h \colon \overline{\mathbb D} \onto \overline{\mathbb D}$ for which
\begin{equation}\label{energyweighted}
\mathcal E_{p,\beta }[h] = \int_{\mathbb D} \frac{\abs{Dh(z)}^p}{(1-\abs{h(z)})^{p\beta}}\, \dtext z < \infty \, .
\end{equation}
\end{theorem}

Note that if $\beta p \le 0$, then  the harmonic extension would simply give the desired homeomorphism~\cite{IMS, Ve}. However, when $0 <\beta p <1$, the harmonic extension does not seem to work and a new way to construct Sobolev homeomorphisms  is  needed.  The above weighted homeomorphic 
extension theorem  gives the requested $\W^{1,p}$-Sobolev homeomorphic 
extension in Question~\ref{existQuestion} for all $p<2$ provided $\Y$ receives an $\alpha$-H\"older continuous  quasiconformal mapping from $\mathbb D$ with $\alpha > \frac{1}{2}$, see Theorem~\ref{thm:nonsense}. 

\begin{question}\label{q:quasi}
Let $\Y \subset \mathbb C$ be a simply connected Jordan domain. Under which  conditions on $\Y$ does there exist a quasiconformal mapping $f \colon \mathbb D \onto \Y$ in $\mathscr C^\alpha (\mathbb D , \C)$ with $\alpha > \frac{1}{2}$?
\end{question}

Recall that it is characteristic for a quasiconformal  mapping to behave locally at every point like a radial stretching, see~\cite{IOZ1}.  Without going into detail, improving the H\"older regularity of $f$ at $x_\circ$ automatically means lowering the H\"older continuity for the inverse map at $y_\circ=f(x_\circ)$.  We expect that a  quasiconformal change of variables in Question~\ref{q:quasi}  (with gained H\"older continuity) exists if the boundary mapping of the  conformal mapping  lies in $\mathscr C^\varepsilon (\partial \mathbb D)$ for some  $\varepsilon >0$. We verify that such a quasiconformal mapping exists if $\Y$ is a quasidisk.

\begin{theorem}\label{thm:holder}
Let $\Y$ be a quasidisk. Then there exists a quasiconformal mapping $f \colon \mathbb D \onto \Y$ in $\mathscr C^\alpha (\mathbb D, \mathbb C)$ with some $\alpha > \frac{1}{2}$.
\end{theorem}

This mapping is obtained by first
constructing a quasisymmetric map from $\partial \mathbb D$ onto $\partial \Y$ which lies in $\mathscr C^\alpha (\partial \mathbb D)$ with $\alpha >\frac{1}{2}$, and then
applying an extension result of P. Tukia~\cite{Tu}. To simplify  the construction of the quasisymmetric map we rely on a result of S. Rohde~\cite{Ro}, which states that any quasicircle
is bilipschitz equivalent to a snowflake-type curve. This allows us to 
assume that $\partial \Y$ is a Rohde-type snowflake  curve. As we have already 
indicated, Theorem~\ref{thm:quasidisk} follows from our weighted extension result and Theorem~\ref{thm:holder}. We will deduce Theorem~\ref{thm:johndisk} to
 Theorem~\ref{thm:quasidisk}.

\section{Definitions}

\begin{definition}\label{def:quasi}
Let $\Omega$ and $\Omega'$ be  planar domains. A homeomorphism $F \colon \Omega \onto \Omega'$ is a {\it quasiconformal} mapping if $F\in \W^{1,1}_{\loc} (\Omega, \C)$ and
there exists a constant $1\le K <\infty$ such that
\[\abs{DF(x)}^2 \le K \det DF(x) \qquad \textnormal{ a.e. in } \Omega \, . \]
\end{definition}
Hereafter $\abs{\cdot}$ stands for the operator  norm of matrices.

\begin{definition}\label{def:john}
A simply connected planar domain $\Y$ with at least two boundary points is a {\it $c$-John disk} if any pair of points $y_1,y_2\in \Y$ can be joined by a rectifiable curve $\gamma \subset \Y$
such that
\begin{equation}
\min_{i=1,2} \ell (\gamma (y_i, y)) \le c \, \dist (y, \partial \Y)
\end{equation}
\end{definition}
Hereafter, $\ell (\gamma (y_i, y)) $ denotes the length of the subcurve of $\gamma$ between $y_i$  and $y$, and   $ \dist (y, \partial \Y)$ is the  distance from  $y$ to the boundary $\partial \Y$. In the case when the value of the constant $c$ plays no role 
we simply say that $\Y$ is a John disk. For equivalent characterizations  we refer to~\cite{NV}.

\section{Proofs of extension results}
In this section we outline the proofs of our main results. The main steps are Theorem \ref{thm:weightedextension} and Theorem \ref{thm:holder}, which are proved in Sections \ref{sec:weightedextension} and \ref{sec:holder} respectively.

\emph{Step 1.} We first show that if Theorem \ref{thm:weightedextension} holds, then any domain which admits a quasiconformal mapping from the unit disk in the H\"older class $C^\alpha (\overline{\dd} , \C)$ for $\alpha > \frac12$ is suitable for extending a given boundary homeomorphism as a Sobolev homeomorphism for $p < 2$.

\begin{proposition}\label{pro:qcholder}

Let $\Y\subset \mathbb C$ be a Jordan domain and $f \colon \overline{\dd} \onto \overline{\Y}$ a homeomorphism. Suppose that $f$ is  a quasiconformal mapping on $\dd$ and  $f \in \mathscr C^\alpha (\overline{ \mathbb D}, \mathbb C)$ for some $\alpha > \frac{1}{2}$. Then there exists a homeomorphism $F\colon \overline{\dd} \onto \overline{\Y}$ which is quasiconformal  on $\dd$ and there is  a constant $C>0$ such that

\begin{equation}\label{eq:stvclaim}
\abs{DF(x)} \le  \frac{C}{(1-\abs{x})^{1-\alpha}}  \qquad \textnormal{ for almost every } x \in \mathbb D \, . 
\end{equation}

\end{proposition}

\begin{proof}

The Sullivan-Tukia-V\"ais\"al\"a approximation theorem~\cite[7.12. Corollary]{TV} provides us with 	a quasiconformal mapping $F \colon \mathbb D \onto \Y$ such that  for any $  \varepsilon  >0 $ we have

\begin{equation}\label{eq:stv1}
k_\Y \big(f(x) , F(x)\big) \le  \epsilon
\end{equation}

and

\begin{equation}\label{eq:stv2}
C^{-1}\, k_\Y \big(F(x) , F(y)\big) \le k_\mathbb D (x,y) \le C \, k_\Y \big(F(x) , F(y)\big) 
\end{equation}

for every $x,y \in \mathbb D$. Here $k_\Omega$ denotes the quasihyperbolic metric in a domain $\Omega$. The constant $C \ge 1$ in~\eqref{eq:stv2} depends on the original mapping $f$  but is  independent of $x$ and $y$. It follows  from~\eqref{eq:stv2} that $F$ is bi-Lipschitz in the quasihyperbolic metrics and locally bi-Lipschitz in the Euclidean metrics and

\begin{equation}\label{eq:stv3}
\abs{DF(x)} \le C \, \frac{\dist (F(x), \partial \Y)}{1-\abs{x}}
\end{equation}  

where $C\ge 1$ is  independent of $x$.  For a proof of~\eqref{eq:stv3} we refer to  e.g~\cite[6.5. Lemma]{TV}. It also follows from~\eqref{eq:stv1} that $f=F$ on $\partial \mathbb D$.

From Lemma 2.1 in~\cite{GP} it follows that

\begin{equation}
\log \left( \frac{\abs{f(x)-F(x)}}{d} +1\right) \le k_\Y (f(x), F(x))
\end{equation}

where $d= \min \{\dist\big (f(x) , \partial \Y \big ), \dist\big (F(x) , \partial \Y \big )  \}$. Combining this with~\eqref{eq:stv1} 

we have

\begin{equation}
\log \left( \frac{\abs{f(x)-F(x)}}{d} +1\right) \le \varepsilon \, .
\end{equation}

Thus for $\varepsilon>0$ small enough we have
\[\abs{f(x)-F(x)} \le (e^\varepsilon -1) d < \frac{d}{2}\]
and thus
\begin{equation}\label{eq:stv4}
\dist \big(F(x), \partial \Y \big) \le 2\, \dist (f(x), \partial \Y) \leq 2C (1-|x|)^\alpha,
\end{equation}
where the last inequality follows from the assumption  $f \in \mathscr C^\alpha (\overline{ \mathbb D}, \mathbb C)$. Now the asserted estimate~\eqref{eq:stvclaim} simply follows from~\eqref{eq:stv3} and ~\eqref{eq:stv4}.

\end{proof}

\begin{theorem}\label{thm:nonsense}
Let $\Y$ be a Jordan domain. Suppose that there is a homeomorphism $f \colon \overline{\dd} \onto \overline{\Y}$ such that $f$ is  quasiconformal on $\dd$ and  $f\in \mathscr C^\alpha (\overline{\dd} , \C)$ for some $\alpha > \frac{1}{2}$. Then every $\varphi \colon \partial \dd \onto \partial \Y$ admits a homeomorphic extension $h \colon \overline{\dd} \onto \overline{\Y}$ of Sobolev class $\W^{1,p} (\dd, \C)$ for all $1\le p <2$.
\end{theorem}
\begin{proof}
Via Proposition \ref{pro:qcholder} we find a quasiconformal map $F$ which takes $\dd$ onto $\yy$ and satisfies
\begin{equation}\label{eq:gradest}
|DF(x)| \leq \frac{C}{(1-|x|)^{1-\alpha}} \qquad \textnormal{ for almost every } x \in \mathbb D.
\end{equation}

Next, we define a homeomorphism  $\psi = F^{-1} \circ \varphi \colon \partial \dd \onto \partial \dd$. Let $\beta=1-\alpha$. Then applying  the extension result Theorem~\ref{thm:weightedextension} gives a Sobolev homeomorphism $H \colon \overline{\dd} \onto \overline{\dd}$ with $H=\psi$ on $\partial \dd$ and
\begin{equation}\label{eq:Henergyest}\mathcal E_{p,\beta }[H] = \int_{\mathbb D} \frac{\abs{DH(z)}^p}{(1-\abs{H(z)})^{p\beta}}\, \dtext z < \infty \, . \end{equation}
Defining $h=F\circ H \colon \overline{\dd} \onto \overline{\Y} $ we have $h=\varphi$ on $\partial \dd$ and since both $F,H$ are locally Lipschitz, $h$ is
locally Lipschitz and  
\[
\int_\dd \abs{Dh(z)}^p \, \dtext z \le \int_\dd \abs{DF\big (H(z)\big )}^p \abs{DH(z)}^p\, \dtext z \, . 
\]
Combining this with~\eqref{eq:gradest} we obtain
\[
\int_\dd \abs{Dh(z)}^p \, \dtext z \le C  \int_{\mathbb D} \frac{\abs{DH(z)}^p}{(1-\abs{H(z)})^{p\beta}}\, \dtext z  \]
and the claim $h\in \W^{1,p} (\dd, \C)$ follows from~\eqref{eq:Henergyest}.
\end{proof}

\emph{Step 2.} Combining the statements of Theorem \ref{thm:nonsense} and Theorem \ref{thm:holder} now proves one of our main results, Theorem \ref{thm:quasidisk}. It remains to show how Theorem \ref{thm:johndisk} then follows from Theorem \ref{thm:quasidisk}.

\begin{proof}[Proof of Theorem~\ref{thm:johndisk}]
 Let $\yy$ be a John disk. Then by~\cite{Br3}, see also \cite{HM}, $\yy$ equipped with the internal metric, which we denote by $d_\yy$, is bilipschitz equivalent to a quasidisk. Hence there exists a quasidisk $\Omega$ and a bilipschitz map $G: (\Omega, |\cdot|) \to (\yy, d_\yy)$. Letting $L$ be the bilipschitz constant of $G$, we find that
\[|G(z_1)-G(z_2)| \leq d_\yy(G(z_1),G(z_2)) \leq L|z_1-z_2|.\]
Thus $G$ is also a Lipschitz mapping in the Euclidean metric. Now, given a boundary homeomorphism $\varphi : \partial\dd \to \partial\yy$, we let $H : \dd \to \Omega$ denote the homeomorphic extension of $G^{-1} \circ \varphi : \partial\dd \to \partial \Omega$ given by Theorem~\ref{thm:quasidisk}. Then the map $h := G \circ H$ gives a homeomorphic extension of $\varphi$ and lies in the Sobolev space $W^{1,p}(\dd)$ for all $p < 2$ since $H$ lies in these spaces and $G$ is Lipschitz. This completes the proof.

\end{proof}

\section{Proof of Theorem~\ref{thm:weightedextension}}\label{sec:weightedextension}

\begin{proof}
We construct the homeomorphic extension $h$ of $\varphi$ as follows. The  given construction can be traced back to  the extension technique of Jerison and Kenig \cite{JEKE}. First, since the unit circle is smooth we may use a bilipschitz map locally to assume it is flat. It will then be enough to give a construction of a 
homeomorphism $H$ from the triangle $T = \{(x,y) \in \R^2 : 0 \leq y \leq 1, y-1 \leq x \leq 1-y \}$ onto itself which is equal to a given boundary homeomorphism $\varphi$ on the real line part of T, the identity mapping on the rest of the boundary and has finite energy 
\[\E^T_{p,\beta}[H] = \int_T \frac{\abs{DH(z)}^p }  {\big[ \dist( h(z)) \big]^{p\beta} }  \, \dtext z < \infty \, . \] 
Here  $\dist(H(z))$ denotes  the distance of $H$ to the real line. We may assume, without loss of generality,  that the boundary  homeomorphism $\varphi \colon [-1,1] \onto [-1,1]$ is increasing. 

The reason why this suffices is that we may first cover the boundary of the unit disc by some finite number of closed intervals disjoint apart from their endpoints. For example, let us cover $\partial \dd$ by four equal length intervals $I_1,\ldots, I_4$ and let $I_j' = \varphi(I_j)$ for each $j$. Then for each of the intervals $I_j$, we connect both endpoints via a line segment to create circular segments $T_1,\ldots,T_4$ over the intervals $I_1,\ldots, I_4$ that are mutually disjoint apart from some endpoints of the $I_j$ and each set is bilipschitz-equivalent to the triangle $T$ via some uniform bilipschitz constant. We do the same to the $I_j'$ to construct four circle segments $T_1',\ldots,T_4'$ which are again mutually disjoint and uniformly bilipschitz-equivalent to the triangle $T$, especially here we use the fact that there are only a finite number of the $I'_j$ to guarantee the fact that the bilipschitz-constant is uniform. The bilipschitz mappings used here may always be chosen such that they map the part on $\partial \dd$ to the real line. See Figure \ref{trianglefig} for an illustration. Once we have shown a way to construct the homeomorphism $H$ as described in the previous paragraph, it is immediate that we obtain a map from $\cup_j T_j$ to $\cup T_j'$ with the correct boundary behaviour, injectivity and energy estimates. It remains to map the square $\dd \setminus \cup_j T_j$ to the quadrilateral $\dd \setminus \cup T_j'$ via a bilipschitz map which is easily constructed since the maps from $T_j$ to $T_j'$ will be shown to be bilipschitz on $\partial T_j \setminus \partial \dd$. On this square the energy of $H$ will be finite since the derivative is bounded from above and the singularity poses no problem since $\dist(H(z))$ is controlled by $\dist(z)$ from below since $H$ is bilipschitz on the square, and $p\beta < 1$.

\begin{figure}[ht]
\includegraphics[scale=0.5]{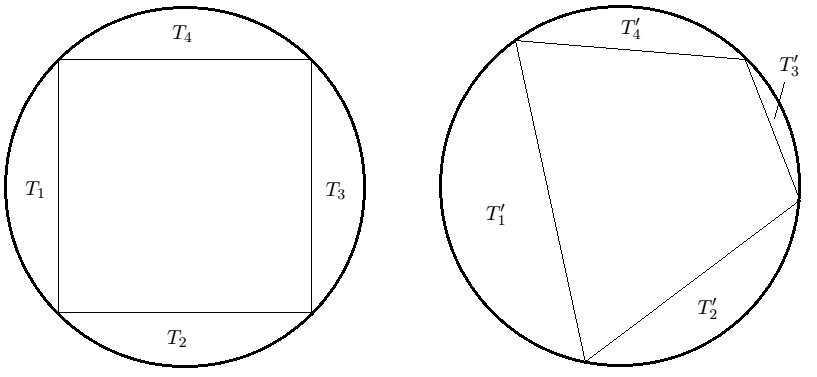}
\caption{The sets $T_j$ and $T_j'$.}\label{trianglefig}
\end{figure}

We now split the boundary interval $[-1,1]$ of $T$ into dyadic intervals $I_{k,j}$, where $j$ denotes the generation of the interval, i.e. $|I_{k,j}| = 2^{-j}$. We denote by $A_{k,j}$ and $B_{k,j}$ the endpoints of such an interval from left to right. For each $k,j$ we also have an image interval $I'_{k,j}$ which is the interval between $\varphi(A_{k,j})$ and $\varphi(B_{k,j})$ on the image side. We shall now construct, for each dyadic interval, a set $U_{k,j}$ that will be mapped onto an image set $U'_{k,j}$. None of these sets will overlap apart from their boundaries and both collections of sets will have union exactly equal to the original triangle $T$.

For an interval $I$ on the real line, we denote by $V(I)$ the point in the upper half plane which, together with the endpoints of $I$, forms an isosceles triangle with base $I$ and right angle at the point $V(I)$, we call this the \emph{apex point} of $I$. Let us now fix a dyadic interval $I_{k,j}$, and we will drop subscripts for the rest of this construction for ease of notation so that $I_{k,j}$ is simply $I$, its image interval $I'_{k,j}$ is simply $I'$ and so on. We define five points $X,Y,z,y,x$ as follows. The first point $X$ is the apex point $V(I)$ of $I$. The second point $Y$ is the apex point of $I_{k+1,j}$, the neighbouring dyadic interval of $I$ from the right. The points $x$ and $y$ are the apex points of the two children of $I$ from left to right. The point $z$ is the apex point of the first child of $I_{k+1,j}$ from the left. Connecting the points $X,Y,x,y,z$ in that order gives a parallelogram with one extra point $y$ on the side between $x$ and $z$ which we call $U$, and we consider this as a pentagon with one straight angle. We now let $X'$ denote the apex point of the image interval $I'$ of $I$. Similarly we define $Y',x',y'$ and $z'$. We now connect the points $X',Y',z',y',x'$ in that order to form a pentagon which we denote by $U'$. See Figure \ref{diamondfig} for an illustration of the configuration. 
\begin{figure}[ht]
\includegraphics[scale=0.35]{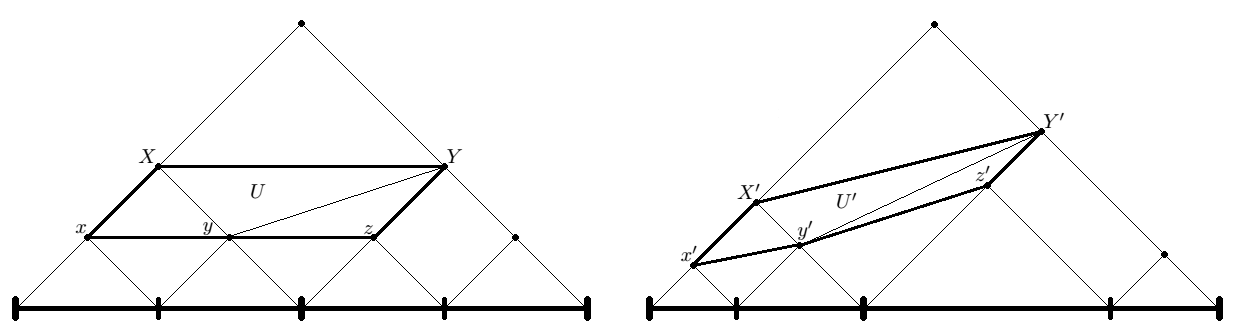}
\caption{The set $U$ and its image set $U'$.}\label{diamondfig}
\end{figure}
We now wish to construct a piecewise affine map from $U$ to $U'$ which is affine on each of the five sides of $U$, mapping each side to the corresponding side of $U'$. Such a construction is easily done if we map the triangles $\Delta Xyx$, $\Delta XYy$ and $\Delta Yzy$ via affine maps to the corresponding triangles on the image side. If the dyadic interval $I$ is the last one of its generation, meaning that there is no neighbouring dyadic interval on the right of it on the interval $[-1,1]$, then we simply define $U$ as the triangle $\Delta Xyx$ and $U'$ as $\Delta X'y'x'$. It is clear that these sets are disjoint apart from their boundaries and fill the triangle $T$ completely. If we define our map $H$ in each of the sets $U$ as described above, then it is also immediate that $H$ is a homeomorphism from $T$ to itself which is equal to the given boundary map $\varphi$ on the real line.


We begin to compute the energy $\mathcal E^T_{p, \beta} [H]$ on $U$. For that, let us start by looking at the map $H$ from $\Delta Xyx$ onto $\Delta X'y'x'$. To make our argument more general, we compute the energy of a general affine map $H$ from $\Delta Xyx$ to a triangle $\Delta P_1P_2P_3$ in the upper half plane. If $L$ denotes the distance of the point $X$ to the real line, then all of the sides of $\Delta Xyx$ are comparable to $L$ and $L$ is comparable to $2^{-j}$, where $j$ is the generation of $I$. Let us define the number $b$ as the largest side length in the triangle $\Delta P_1P_2P_3$. Hence the norm of the differential of the map $H$ on $\Delta Xyx$ can be estimated from above by a constant times $b/L$ since the angles in $\Delta Xyx$ are bounded from below by some positive constants. Thus
\begin{equation}\label{triangleintegral}
\int_{\Delta Xyx} \frac{\abs{DH(z)}^p}{\dist(H(z))^{p\beta}}\, \dtext z \leq C\frac{b^p}{L^p}\int_{\Delta Xyx} \frac{1}{\dist(H(z))^{p\beta}}\, \dtext z.
\end{equation}
It remains to estimate the integral expression on the right-hand side. Clearly we may assume that $p\beta >0$.  Suppose without loss of generality that the distance from $P_1$ to the real line is larger than or equal  to the distance from $P_2$ and $P_3$ to the real line, and call this largest distance $\tilde{b}$. If we now modify the triangle $\Delta P_1P_2P_3$ by moving both points $P_2$ and $P_3$ to the real line, then the integral on the right-hand side of \eqref{triangleintegral} increases. In this case, the quantity $\dist(H(z))$ is equal to $\tilde{b}$ on one vertex of the triangle $\Delta Xyx$ and equal to zero on the opposing side of the triangle. To make the computation a bit easier, we cover this triangle with a rectangle $R$ whose both side lengths are comparable to $L$, one side of $R$ contains the side of the triangle where $\dist(H(z))$ is zero and the opposing side of $R$ contains the vertex where $\dist(H(z))$ is equal to $\tilde{b}$. See Figure \ref{diamondfigy} for an illustration.
\begin{figure}[ht]
\includegraphics[scale=0.35]{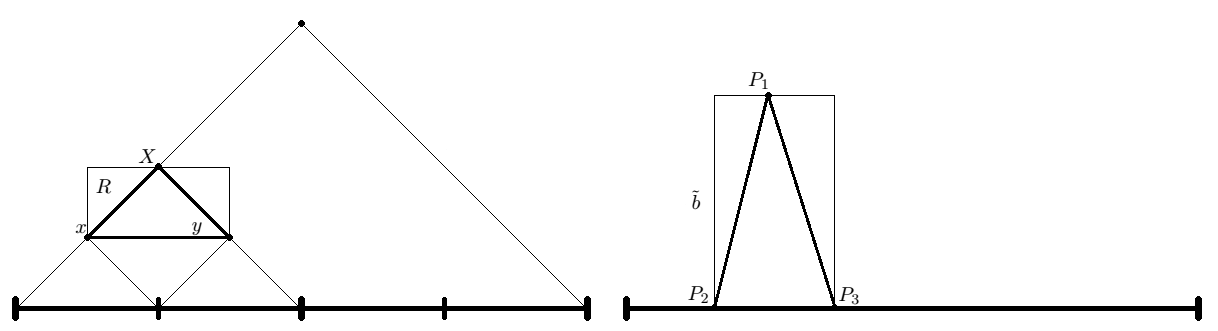}
\caption{The map $H$ on the rectangle $R$.}\label{diamondfigy}
\end{figure}
Letting the side lengths of $R$ be $c_1L$ and $c_2L$, we compute that
\begin{align*}
\int_{R} \frac{1}{\dist(H(z))^{p\beta}}\, \dtext z
&\leq C_1\int_0^{c_1L}\int_0^{c_2L} \frac{1}{\left(\frac{\tilde b}{L} y\right)^{p\beta}}\, \dtext x \, \dtext y
\\&= C_2\frac{L^{1+p\beta}}{\tilde b^{p\beta}} \int_0^{c_1L} \frac{1}{y^{p\beta}}\, \dtext y
\\&= \frac{C_3}{1-p\beta}\frac{L^2}{\tilde{b}^{p\beta}}.
\end{align*}
Here we also used the assumption $p \beta <1$.
Hence for the full energy of $H$ over $\Delta Xyx$ we obtain the estimate
\[\int_{\Delta Xyx} \frac{\abs{DH(z)}^p}{\dist(H(z))^{p\beta}}\, \dtext z  \leq C L^{2-p}\frac{b^p}{\tilde{b}^{p\beta}},\]
where $C$ only depends on $p$ and $\beta$. 

Let us now suppose that the arbitrary triangle $\Delta P_1P_2P_3$ was $\Delta X'y'x'$. In the triangle $\Delta X'y'x'$, the maximum distance from each vertex to the real line is comparable to the maximum side length because of the way this triangle was constructed via apex points of the dyadic intervals. Hence we find that
\[\int_{\Delta Xyx} \frac{\abs{DH(z)}^p}{\dist(H(z))^{p\beta}}\, \dtext z  \leq C L^{2-p}\, b^{p-p\beta},\]
where $b$ again denotes the maximum side length of the triangle $\Delta X'y'x'$ or equivalently the maximum length of the dyadic intervals whose apex points are vertices of this triangle.
\\\\
The calculation in the triangles $\Delta XYy$ and $\Delta Yzy$ is now exactly the same. In each of these triangles the side lengths are comparable to the same $L$ as before, and the angles are controlled from below. The energy of $H$ on both of these triangles is again estimated from above by the maximum side length of the corresponding target triangle, which is always comparable to the maximum distance to the real line in the target triangle. Thus we find the same estimate
\begin{equation}\label{deltaestim}\int_{\Delta} \frac{\abs{DH(z)}^p}{\dist(H(z))^{p\beta}}\, \dtext z  \leq C L^{2-p}\, b_\Delta^{p-p\beta},\end{equation}
where $\Delta$ denotes one of the triangles $\Delta XYy$ or $\Delta Yzy$ and $b_\Delta$ denotes the maximum side length of the corresponding target triangle, which is again comparable to the maximum length of the involved dyadic intervals. If the original dyadic interval was $I_{k,j}$, then these maximum lengths of the corresponding target triangles over each of the three triangles that make up the corresponding set $U$ are estimated from above by the quantity $|I'_{k,j}| + |I'_{k+1,j}|$. Applying this to \eqref{deltaestim} gives that
\[\int_{U} \frac{\abs{DH(z)}^p}{\dist(H(z))^{p\beta}}\, \dtext z  \leq C L^{2-p}\left(|I'_{k,j}|^{p-p\beta} + |I'_{k+1,j}|^{p-p\beta}\right).\]
Summing over all of the dyadic intervals, we find that
\[\int_{T} \frac{\abs{DH(z)}^p}{\dist(H(z))^{p\beta}}\, \dtext z  \leq C \sum_{j=0}^\infty 2^{-j(2-p)}\sum_{k=1}^{2^j}|I'_{k,j}|^{p-p\beta}.\]
We aim to show that the double sum on the right-hand side is finite. We consider first the case $p(1-\beta) \geq 1$. In this case
\begin{align*}\sum_{j=1}^\infty 2^{-j(2-p)}\sum_{k=1}^{2^j} |I'_{k,j}|^{p(1-\beta)} &\leq \sum_{j=1}^\infty 2^{-j(2-p)}\left(\sum_{k=1}^{2^j} |I'_{k,j}|\right)^{p(1-\beta)} \\&= \sum_{j=1}^\infty 2^{-j(2-p)} \cdot 2^{p(1-\beta)} < \infty\end{align*}
since $\sum_{k=1}^{2^j} |I'_{k,j}|$ is just the total length of the target boundary on the real line, which we assume to be equal to $2$. If $p(1-\beta) < 1$, then by H\"older's inequality
\[\sum_{k=1}^{2^j} |I'_{k,j}|^{p(1-\beta)} \leq 2^{j(1-p(1-\beta))}\left(\sum_{k=1}^{2^j} |I'_{k,j}|\right)^{p(1-\beta)} = 2^{j(1-p(1-\beta))}\cdot 2^{p(1-\beta)}.\]
Hence
\[\sum_{j=1}^\infty 2^{-j(2-p)}\sum_{k=1}^{2^j} |I'_{k,j}|^{p(1-\beta)} \leq 2^{p(1-\beta)}\sum_{j=1}^\infty 2^{-j(1-p\beta)} < \infty.\]
Thus $\mathcal E^T_{p, \beta} [H]< \infty$, which completes the proof.
\end{proof}
\section{Proof of Theorem~\ref{thm:holder}}\label{sec:holder}

\begin{proof}
Let us first invoke a result of Rohde \cite{Ro} which states that any quasicircle is bilipschitz equivalent to a snowflake-type curve. This allows us to assume that $\Gamma$ is a snowflake-type curve. We shall briefly explain the definition of such a snowflake-type curve.

To construct a snowflake-type curve $S$, we fix a parameter $p \in [1/4,1/2)$. Let us then construct a sequence of curves $(S_n)$ as follows. Let $S_0$ denote the unit square in the plane, and let us call its sides the \emph{segments} of $S_0$. We now construct the sequence $(S_n)$ inductively. For each segment $s$ in $S_n$, there are two choices. We replace the segment with a translated and scaled copy one of the two choices in Figure \ref{fig4}. In any case, the segment $s$ has been replaced by four smaller segments, which we call the \emph{children} of $s$. The curve obtained by making this choice for each segment $s$ in $S_n$ will be the curve $S_{n+1}$. We assume from now on that all of the choices in the construction have been fixed. The collection of all segments in all of the curves $S_n$ is denoted by $\mathcal{P}$.

\begin{figure}[ht]
\includegraphics[scale=0.5]{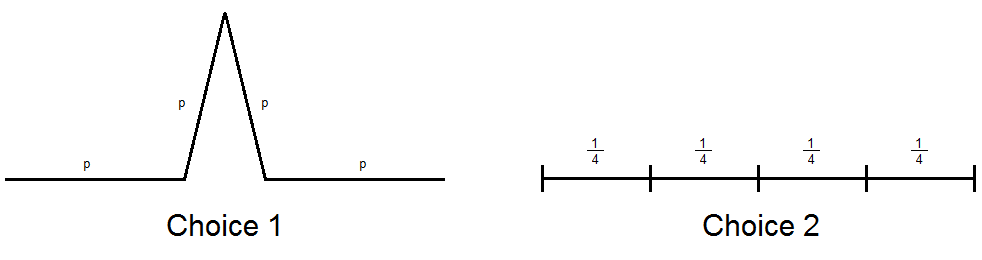}
\caption{The two choices of how to replace a segment in $S_n$. The quantities in the picture represent the portion of the total length of the segment.}\label{fig4}
\end{figure}

Every possible sequence of choices leads to a different sequence of curves $S_n$, but in any case these curves will converge to a limit curve $S$. Thus the snowflake-type curves are defined as limit curves of these kinds of constructions. In order to define the required quasiconformal map to the bounded Jordan domain whose boundary is $S$, it is enough to find a quasisymmetric boundary map $g: S_0 \to S$ such that $g \in C^{\alpha}$ with $\alpha > 1/2$. This is due to an extension theorem of Tukia, see~\cite{Tu1} and Theorem 2.1 in \cite{Kholder} for a convenient formulation of the result used here. We shall now explain the convergence of $S_n$ to $S$ more in detail as we would like to fix a parametrization $g_n : S_0\to S_n$ so that we obtain the desired map $g$ as a limit.

Some terminology used here: By two disjoint line segments we mean that they share at most one point (we do not pay much mind to whether line segments are open or closed). Two quantities are comparable (denoted $\approx$) if both can be estimated by a constant times the other. The dependence of the constant will be only on the parameter $p$, unless explicitly stated otherwise.

Let the exponent $\alpha$ and the number $x$ be defined by the equations
\begin{equation}\label{relations}
\left(\frac{1}{4}\right)^{\alpha} = p \qquad \text{ and } \qquad x^{\alpha} = \frac{1}{4}.
\end{equation}
Hence $\alpha > 1/2$ and $x \leq 1/4$. Let us now construct the sequence $(g_n)$ inductively. We let $g_0 : S_0 \to S_0$ be the identity map. We then construct $g_{n+1}$ based on $g_n$. For each segment $s$ in $S_n$, let $I$ be the preimage of $s$ under $g_n$ which will always be a line segment. If the segment $s$ was split according to Choice 1 in Figure \ref{fig4}, then we split $I$ into four equal length line segments and define $g_{n+1}$ so that it maps each of these line segments to the children of $s$ linearly, see Figure \ref{fig5}. If instead $s$ was split according to Choice 2, then we split $I$ as in Figure \ref{fig5} into two segments of length $x$ and two segments of length $1/2-x$. These will be mapped to the children of $s$ as in Figure \ref{fig5}. The four intervals that $I$ splits into are also called the children of $I$. Let also $\mathcal{R}$ denote the collection of all such line segments $I$ which are preimages of some segment in $\mathcal{P}$ under the appropriate $g_n$. This also induces a natural map $g^* : \mathcal{P} \to \mathcal{R}$.

\begin{figure}
\includegraphics[scale=0.5]{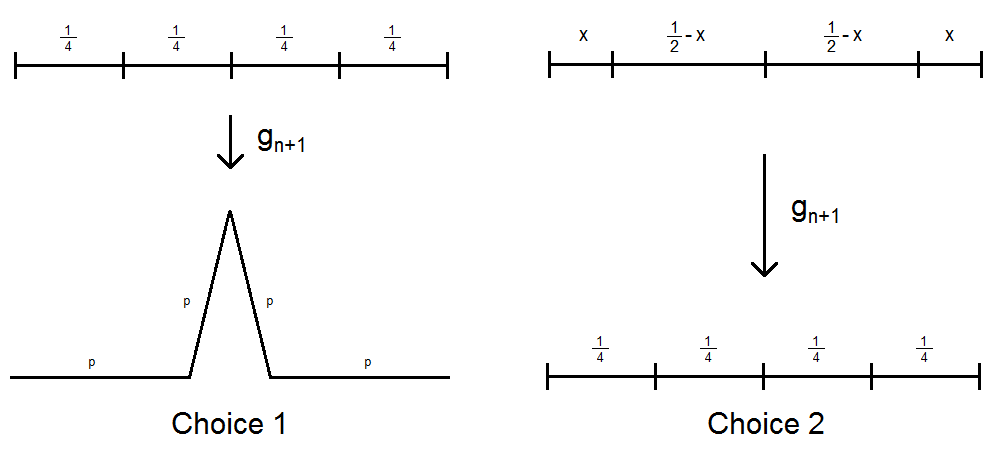}
\caption{How to construct $g_{n+1}$ in the two different cases. The quantities in the picture represent the portion of the total length of the segment.}\label{fig5}
\end{figure}
One may verify that the mappings $g_n$ converge uniformly to a homeomorphism $g: S_0 \to S$. We now aim to show the H\"older-continuity and quasisymmetry of the map $g$.\\\\
%
Let us first explain how to calculate the length $\ell(s)$ of a line segment $s$ in $\mathcal{R}$ or $\mathcal{P}$. Denote by $\mathcal{F}$ the collection of all finite words that can be formed using the letters $A, B$ and $C$. We now define a map $\tau$ from $\mathcal{R}$ to $\mathcal{F}$ inductively as follows. If $I$ is one of the sides of $S_0$, then $\tau(I)$ is the empty word. If for some $I \in \mathcal{R}$ we have already defined $\tau(I) = w$ for a word $w \in \mathcal{F}$, then $\tau$ will be defined on the children of $I$ as follows. If the children of $I$ are formed via Choice 1 in Figure \ref{fig5}, then we define $\tau(I') = wA$ for every child $I'$ of $I$, where $wA$ denotes the word obtained by adding the letter $A$ to the end of $w$. If instead the children are formed based on Choice 2, then $\tau(I') = wB$ for those children $I'$ of $I$ for which $\ell(I')/\ell(I) = x$ and $\tau(I') = wC$ for those children for which $\ell(I')/\ell(I) = 1/2-x$.

Let now $a(w)$ denote the number of letters $A$ in the word $w\in\mathcal{F}$, similarly $b(w)$ the number of letters $B$ and $c(w)$ the number of letters $C$. Then from the construction we find the formulas
\[\ell(I) = \left(\frac{1}{4}\right)^{a(\tau(I))} x^{b(\tau(I))} \left(\frac{1}{2}-x\right)^{c(\tau(I))}\qquad\text{ for all } I \in \mathcal{R}\]
and
\[\ell(s) = p^{a(\tau(g^*(s)))} \left(\frac{1}{4}\right)^{b(\tau(g^*(s)))+c(\tau(g^*(s)))}\qquad\text{ for all } s \in \mathcal{P}.\]
Notice that by the relations of $x,\alpha$ and $p$ in \eqref{relations}, we have for every $s \in \mathcal{P}$ that
\begin{equation}\label{etaEstim}
\frac{\ell(s)}{\ell(g^*(s))^\alpha} = \eta^{c(\tau(g^*(s)))} \quad \text{ where } \eta = \frac{1/4}{(\frac{1}{2}-x)^\alpha} < 1.
\end{equation}
Hence
\begin{equation}\label{segmentEstim}\ell(s) \leq \ell(g^*(s))^\alpha\end{equation} for all such $s$. We define another function $\mu$ on $\mathcal{P}$ which sends every segment $s$ to the smaller arc of the snowflake-type curve $S$ with the same endpoints as $s$. From the construction of the snowflake-type curve one may see that the diameters of $s$ and $\mu(s)$ are always comparable. Then \eqref{segmentEstim} implies that
\begin{equation}\label{gHolder} \diam(g(I)) \leq C\ell(I)^\alpha \quad \text{ for all } I \in \mathcal{R}.\end{equation}

Let now $J$ be any arc of $S_0$. Take a cover of $J$ with line segments from $\mathcal{R}$ with disjoint interiors so that the number of line segments in this cover is minimal. Then there cannot be more than six line segments in this cover since in any seven consecutive line segments there are always four which are exactly the children of another line segment in $\mathcal{R}$ (with which we could then replace these four). Furthermore, we may choose the line segments so that their length is at most a constant depending on $p$ times the total length of $J$. From \eqref{gHolder} we then find that $g$ must be H\"older-continuous of exponent $\alpha$.

We must now show that $g$ is quasisymmetric. Thus we must prove that
\[\frac{1}{C} \leq \frac{|g(x+t)-g(x)|}{|g(x) - g(x-t)|} \leq C\] for some constant $C$. Due to the nature of the construction, if we define $J_+ = [x,x+t]$ and $J_- = [x-t,x]$ then it holds that $|g(x+t)-g(x)|$ is always comparable to $\diam(g(J_+))$ and similarly $|g(x)-g(x-t)|$ is comparable to $\diam(g(J_-))$. Now note that any arc $J$ in $S_0$ may be covered by six or less line segments from $\mathcal{R}$ of comparable length with $J$ and must also contain at least one line segment from $\mathcal{R}$ of comparable length with $J$, which shows that it is enough to prove the following claim to deduce the quasisymmetry of $g$.\\\\
\textbf{Claim.} Let $C$ be a fixed constant. Suppose $I_1,I_2 \in \mathcal{R}$ are disjoint intervals such that $C^{-1} \ell(I_1) \leq \ell(I_2) \leq C\ell(I_1)$ and $\dist(I_1,I_2) \leq C \ell(I_1)$. Then the lengths of $g(I_1)$ and $g(I_2)$ are comparable by a constant only depending on $C$ and $p$.\\\\
Suppose without loss of generality that $c(\tau(I_1)) \geq c(\tau(I_2))$. By formula \eqref{etaEstim}, the lengths of $g(I_1)$ and $g(I_2)$ differ by at most a constant times $\eta^N$, where $N := c(\tau(I_1))-c(\tau(I_2))$. Hence we are to estimate the number $N$. Denote by $I_1^{(1)}$ the parent of $I_1$, $I_1^{(2)}$ the parent of $I_1^{(1)}$ and so on. We consider the line segment $I^* = I_1^{(N-2)}$. This choice implies that
\begin{equation}\label{cEquation}
c(\tau(I^*)) \geq c(\tau(I_2))+2,
\end{equation} which shows that the words $\tau(I^*)$ and $\tau(I_2)$ differ by at least two letters $C$. The line segments $I^*$ and $I_2$ must be disjoint since otherwise one would contain the other, which would either imply $I_1 \subset I_2$ or contradict \eqref{cEquation}.

Suppose first that the word $\tau(I_2)$ has at least as many letters as $\tau(I^*)$, meaning that $I_2$ is of the same or later generation than $I^*$. In this case let $I_2^*$ be the line segment in $\mathcal{R}$ which contains $I_2$ and such that $\tau(I^*)$ and $\tau(I_2^*)$ have the same length. Since $\tau(I_2^*)$ has fewer letters than $\tau(I_2)$, we have $c(\tau(I^*)) \geq c(\tau(I_2^*))+2$ by \eqref{cEquation}. Hence there must be a line segment $I_3^*$ in $\mathcal{R}$ between $I^*$ and $I_2^*$ of the same generation. We may assume this line segment is a neighbor of $I^*$. Now we must have that $\ell(I_3^*) \leq \dist(I_1,I_2) \leq C\ell(I_1)$. Furthermore,
\[\ell(I_3^*) \approx \ell(I^*) \geq \frac{\ell(I_1)}{(1/2-x)^{N-2}},\]
since taking the parent of a line segment increases the length by at least a factor of $(1/2-x)^{-1}$. This gives a bound for $N$ in terms of $C$ and $p$, which is enough.

Suppose now that the word $\tau(I_2)$ has less letters than $\tau(I^*)$. Let $I_2^*$ be the line segment in $\mathcal{R}$ which is contained in $I_2$, of the same generation as $I^*$, and closest to $I^*$ (hence sharing an endpoint with $I_2$). Then by construction we have $c(\tau(I_2^*)) = c(\tau(I_2)) \leq c(\tau(I^*))-2$. The rest of the proof follows the same line of arguments as the previous case (starting from the definition of $I_3^*$). This proves the claim.
\end{proof}

\end{document}